\documentclass[11pt]{amsart}

\usepackage{centernot}

\renewcommand{\phi}{\varphi}
\newcommand{\GG}{\mathrm{G}}

\newcommand{\ds}{\displaystyle}
\newcommand{\ben}{\begin{enumerate}}
\newcommand{\een}{\end{enumerate}}
\newcommand{\eq}[2][label]{\begin{equation}\label{#1}#2\end{equation}}
\newcommand{\av}[2]{\langle #1\rangle_{_{\scriptstyle #2}}}

\usepackage{amsfonts}
\usepackage{amssymb}
\usepackage{amsmath}
\usepackage{graphicx,color}

\newcommand{\bel}[1]{\boldsymbol{#1}}
\newcommand{\df}{\stackrel{\mathrm{def}}{=}}

\newcommand{\BMO}{{\rm BMO}}

\newcommand{\kf}{K_{h,Q}(\varphi)}

\newtheorem{theorem}{Theorem}[section]

\newtheorem{lemma}[theorem]{Lemma}

\newtheorem*{theorem*}{Theorem}{\bf}{\it}
\newtheorem*{proposition*}{Proposition}{\bf}{\it}
\newtheorem*{observation*}{Observation}{\bf}{\it}
\newtheorem*{lemma*}{Lemma}{\bf}{\it}

\theoremstyle{definition}

\theoremstyle{remark}
\newtheorem{remark}[theorem]{Remark}

\numberwithin{equation}{section}

\allowdisplaybreaks

\begin{document}
\title{Weak integral conditions for BMO}
\author{ A. A. Logunov}
\author{ L. Slavin}
\author{ D. M. Stolyarov}
\author{V. Vasyunin}
\author{P. B. Zatitskiy}



\date{August 12, 2013}

 \begin{abstract}
 We study the question of how much one can weaken the defining condition of BMO. Specifically, we show that  if $Q$ is a cube in $\mathbb{R}^n$ and $h:[0,\infty)\to[0,\infty)$ is such that $h(t)\underset{t\to\infty}{\longrightarrow}\infty,$ then
 $$
\sup_{J~\text{subcube}~Q}\frac1{|J|}\int_J h\big(\big|\varphi-\textstyle{\frac1{|J|} \int_J\varphi} \big|\big)<\infty\quad\Longrightarrow\quad \varphi\in\BMO(Q).
 $$
Under some additional assumptions on $h$ we obtain estimates on $\|\varphi\|_{\BMO}$ in terms of the supremum above. We also show that even though the condition $h(t)\underset{t\to\infty}{\longrightarrow}\infty$ is not necessary for this implication to hold, it becomes necessary if one considers the dyadic $\BMO.$ 
\end{abstract}

\maketitle
\section{Introduction and main results} 
\label{preliminaries}
Let $|E|$ denote the Lebesgue measure of a measurable set $E\subset \mathbb{R}^n.$ If $0<|E|<\infty,$ we use the symbol $\av{\varphi}E$ for the average of a locally integrable function $\varphi$ over $E,$ $\av{\varphi}E=\frac1{|E|}\int_E\varphi.$

Our main object of interest is the space $\BMO(\mathbb{R}^n),$ first described by John and Nirenberg~\cite{jn}:
\eq[bmo1]{
\BMO(\mathbb{R}^n)=\{\varphi\in L^1_{loc}\colon \sup_{\text{cube}~J}\av{|\varphi-\av{\varphi}J|}J<\infty\}.
}

Here the supremum is taken over all cubes $J$ in $\mathbb{R}^n$ with sides parallel to the coordinate axes. We will use the symbol $\BMO(Q)$ when $J$ is restricted to be a subcube of a given cube $Q$ or simply write $\BMO$ when the context is clear or inconsequential.

It is known that for any $p>0$ the mean oscillations in~\eqref{bmo1} can be replaced with mean $p$-oscillations:
\eq[bmo2]{
\BMO=\{\varphi\in L^1_{loc}\colon \sup_{\text{cube}~J}\av{|\varphi-\av{\varphi}J|^p}J<\infty\}.
}
Our preferred definition is the one with $p=2,$ and we reserve the symbol $\|\varphi\|_{\BMO}$ for this choice of $p:$
$$
\|\varphi\|^{}_{\BMO}=\sup_{\text{cube}~J}\Big(\av{|\varphi-\av{\varphi}J|^2}J \Big)^{1/2}=\sup_{\text{cube}~J}\Big(\av{\varphi^2}J-\av{\varphi}J^2\Big)^{1/2}.
$$
The BMO condition~\eqref{bmo1} self-improves even further, to local exponential integrability, a fact that is quantified by the John--Nirenberg inequality 
(see \cite{jn}).

In this note, we investigate the reverse question: How far can one weaken the defining condition of BMO? Specifically,  if $h$ is a non-negative function on $[0,\infty)$ and $\varphi$ an integrable function on a cube $Q,$ let
\eq[e0]{
\kf=\sup_{J~\text{subcube}~Q}\av{h(|\varphi-\av{\varphi}J|)}J.
}
What conditions on $h$ would ensure that if $\kf<\infty,$ then $\varphi\in\BMO(Q)?$ (Or, in shorthand: ``For what $h$ does $K_h$ imply BMO?'').

The question of defining BMO by weak conditions is now classical, dating back to John~\cite{john} and Str\"omberg~\cite{stromberg}. The formulation closest to ours was studied by Long and Yang~\cite{ly}. They proved that if $h$ is a non-negative, increasing, and continuous function on $[0,\infty)$ such that 
$$
\lim_{t\to\infty}h(t)=\infty\quad\text{and}\quad h(s+t)\leqslant h(s)+h(t)+C,
$$
then a function satisfying the $K_h$ condition also satisfies the John--Nirenberg inequality and thus is in BMO. (They studied this question in spaces of homogeneous type, while attributing the result for $\mathbb{R}^n$ to Fang and Wang, though we have not been able to locate that paper.) This result was further generalized by Shi and Torchinsky in~\cite{st}.

We prove a result similar to those in~\cite{ly} and~\cite{st}, but out proof uses a different, novel technique, which we trust is particularly well suited for working with BMO. Our focus is on explicit estimates, relating the BMO norm of a function and its $K_h$ ``norm.'' As far as we know, the norm estimates we present are new in all dimensions, but we also prove a separate, much refined result in dimension 1. As a corollary to the main theorem, we obtain an even stronger qualitative result, one that removes all conditions on $h$ except for the limit condition (this, second result, although of independent interest, could also be derived from Long and Yang's theorem). Here are our main theorems.

\begin{theorem}
\label{t0}
Let $h$ be a function continuous on $[0,\infty)$ and thrice differentiable on $(0,\infty)$ such that 
\eq[c1]{
h(0)=0,\quad h(t)\underset{t\to\infty}{\longrightarrow}\infty,
} 
and for all $t>0$
\eq[c2]{
h'(t)>0,~h''(t)<0,~h'''(t)>0.
}
Let $Q$ be a cube in $\mathbb{R}^n.$  Assume that a function $\varphi\in L^1(Q)$ satisfies $\kf<\infty.$ Then $\varphi\in\BMO(Q)$ and
\eq[ms]{
h^{-1}\Big(\kf\Big)\leqslant\|\varphi\|^{}_{\BMO(Q)}\leqslant 2^{-n/2-1}h^{-1}\Big(2^{n+2} \kf\Big),
}
where $h^{-1}$ is the inverse function to $h$ on the interval $[0,\infty).$
\end{theorem}
\begin{remark}
\label{rr}
The left-hand inequality in~\eqref{ms} is elementary: because $h$ is concave and increasing, we have, for any cube $J,$
$$
\av{h(|\varphi-\av{\varphi}J|)}J\leqslant h(\av{|\varphi-\av{\varphi}J|}J)\leqslant h(\|\varphi\|^{}_{\BMO}),
$$
and we can now take the supremum over all $J$ on the left and then invert $h.$ Thus of principal interest here is the right-hand estimate in~\eqref{ms}.
\end{remark}
\begin{remark}
It is easy to show that any increasing, concave function $h$ on $[0,\infty)$ such that $h(0)=0$ automatically satisfies the triangle inequality:
\eq[e1]{
h(|s+t|)\leqslant h(|s|)+h(|t|), \quad\forall s,t\in\mathbb{R}.
}
\end{remark}

Our method allows us to obtain a much stronger version of Theorem~\ref{t0} in the case $n=1.$
\begin{theorem}
\label{t0.5}
Let $h$ be as in Theorem~\ref{t0}. Let $Q$ be an interval and assume that $\varphi$ is a non-constant\textup, locally integrable function on $Q$ that satisfies $\kf<\infty.$ Then $\varphi\in\BMO(Q)$ and for any subinterval $J$ of $Q$ we have the following sharp inequality\textup:
\eq[uu2]{
\frac{\av{\varphi^2}J-\av{\varphi}J^2}{4\|\varphi\|^2_{\BMO(Q)}}\,h(2\|\varphi\|^{}_{\BMO(Q)})\leqslant \av{h(|\varphi-\av\varphi J|)}J.
}
Consequently\textup,
\eq[uu3]{
h^{-1}\big(\kf\big)\leqslant\|\varphi\|^{}_{\BMO(Q)}\leqslant\frac12\,h^{-1}\big(4\kf\big).
}
\end{theorem}
\begin{remark}
We note that while inequality~\eqref{uu2} is sharp, in that there exists a non-constant function $\varphi$ for which it becomes an equality, the resulting norm inequality~\eqref{uu3} may not be sharp. This phenomenon is explained in~\cite{svas}; it is due to the fact that the suprema in $\|\varphi\|_{\BMO(Q)}$ and $\kf$ are, in general, attained on different subintervals of $Q.$

Nonetheless,~\eqref{uu3} gives the best known bounds for the BMO norm of $\varphi$ in terms of $\kf.$ In particular, setting $h(t)=t^p$ for $0<p<1$, we get the norm equivalence inequality of~\cite{svas}:
$$
\|\varphi\|^{}_{\BMO^p(Q)}\leqslant \|\varphi\|^{}_{\BMO(Q)}\leqslant 2^{2/p-1}\,\|\varphi\|^{}_{\BMO^p(Q)},
$$
where $\|\varphi\|^{}_{\BMO^p(Q)}$ is the supremum in~\eqref{bmo2} raised to the power $1/p.$

Setting $h(t)=\log(1+t)$ provides an appropriate analog of $\BMO^p$ for $p=0$ (note that even though 
$(\av{|\varphi-\av{\varphi}J|^p}J)^{1/p}\to \exp{\av{\log|\varphi-\av{\varphi}J|}J}$ as $p\to0$, the function $\log t$ is not a suitable candidate for this role, as it changes sign on $(0,\infty)$). For this choice of $h$ we get
$$
e^{\kf}-1\leqslant\|\varphi\|^{}_{\BMO(Q)}\leqslant \frac12(e^{4\kf}-1).
$$

\end{remark}

Our next theorem reflects the fact that an (almost) arbitrary $h$ can be modified to fit the hypotheses of Theorem~\ref{t0}, while largely preserving the $K_h$ condition.

\begin{theorem} \label{t1}
Let $h\colon [0,\infty) \to [0,\infty)$ be a measurable function such that
\eq[h1]{
h(t)\underset{t\to\infty}{\longrightarrow}\infty.
}
Let $Q$ be a cube in $\mathbb{R}^n$ and let $\varphi\in L^1(Q)$ be such that $K_{h,Q}(\varphi)<\infty.$ Then $\varphi \in \BMO(Q)$.
\end{theorem}

As mentioned earlier, the qualitative result of  Theorem~\ref{t0} is not new. In fact, this theorem can be obtained as a corollary of a key lemma due to John (see~\cite{john}, p.~469; see Str\"omberg~\cite{stromberg} for a sharp version). Of principal interest in this paper is our method of proof. Traditional proofs of BMO inequalities use stopping time arguments, such as the Calder\'on--Zygmund decomposition in its various forms. Our proof uses Bellman functions instead, which is, very roughly, a technique that estimates integral functionals by second-order variational calculus. As a result, 
we obtain integral estimates directly, without having to bound the distribution function of the BMO function in question. Moreover, we can expand the range of applicability of our theorem by improving a single block in its proof. Indeed, the conditions on $h$ are largely those that allow us to solve the resulting Bellman PDE; as the method develops, we expect to be able to handle very general classes of functions. One often sees Bellman functions in conjunction with sharp results. Our results are only sharp in dimension~1 (Theorem~\ref{t0.5}), but they are wholly new in this generality and precision.

We prove Theorem~\ref{t0} by considering a dual problem: rather than assuming the $K_h$ condition on a function $\varphi$ and then estimating $\|\varphi\|_{\BMO(Q)}$ from above in terms of $\kf,$ we estimate $\kf$ {\it from below} in terms of $\|\varphi\|_{\BMO(Q)}$ and then invert the resulting estimate (this idea in the context of BMO was first used in~\cite{slavin}). The quantities involved are not {\it a priori} finite for an arbitrary $\varphi$, and so we first prove our inequalities for dyadic-simple functions $\varphi$ and then employ an approximation argument. To prove the main estimate, we pose an extremal problem and present an appropriate substitute for its difficult-to-find solution (we call that substitute a sub-solution, since it provides a lower estimate). The sub-solution is then used in an inductive argument that yields the desired inequality. 

The Bellman approach to problems on BMO was first implemented in the paper~\cite{sv} on the integral John--Nirenberg inequality. The Bellman treatment of the classical John--Nirenberg inequality can be found in~\cite{vv}. In~\cite{svas}, the general Bellman theory of integral estimates on BMO was initiated in the context of sharp $L^p$ inequalities for BMO functions. That project was much developed in~\cite{bel_ref_short} and~\cite{bel_ref_staraya}. The latter paper supplies the main ingredient we use here (the extremal sub-solution). However, our presentation does not go into the details of its origin. We simply verify that it has the properties we need and then use it in induction. 

The paper is organized as follows: in Section~\ref{pt0} we prove Theorem~\ref{t0}, save for the main estimate, Lemma~\ref{bel}, whose proof is presented in Section~\ref{Bellman_technique}. Sections~\ref{pt0.5} and~\ref{pt1} give the proofs of Theorems~\ref{t0.5} and~\ref{t1}, respectively. Lastly, in Section~\ref{nec}, we consider the question of  whether the condition $\lim_{t\to\infty} h(t)=\infty$ is necessary for $K_h$ to imply $\BMO.$ It turns out that the answers are different for the usual BMO and its dyadic analog, $\BMO^d.$

\section{Proof of Theorem~\ref{t0}}
\label{pt0}
The proof is a reduction to the case of dyadic BMO for which a key estimate is obtained using a Bellman function. As explained in Remark~\ref{rr} we need to prove only the right-hand inequality in~\eqref{ms}.

For a cube $Q\subset\mathbb{R}^n$ and an integer $m\geqslant0,$ let $D(Q)$ be the set of all dyadic subcubes of $Q$ and $D_m(Q)=\{J\in D(Q), |J|=2^{-mn}|Q|\}.$ For $\varphi\in L^1(Q),$ let $\varphi_m$ be its dyadic truncation of order $m:$
$$
\varphi_m=\sum_{J\in D_m(Q)}\av{\varphi}J\chi^{}_J.
$$ 
If $\varphi=\varphi_m$ for some $m,$ we call $\varphi$ {\it dyadic-simple} on $Q.$

We will use the dyadic BMO,
$$
\BMO^d(Q)=\Big\{\varphi\in L^1(Q)\colon \|\varphi\|_{\BMO^d(Q)}^2\df\sup_{J\in D(Q)}\big(\av{\varphi^2}J-\av{\varphi}J^2\big)<\infty\Big\},
$$
and the dyadic analogue of functional~\eqref{e0},
$$
K^d_{h,Q}(\varphi)=\sup_{J\in D(Q)}\av{h(|\varphi-\av{\varphi}J|)}J.
$$
It is clear that $\|\varphi_m\|_{\BMO^d(Q)}\leqslant \|\varphi\|_{\BMO^d(Q)}$ for all $m.$ A slightly weaker inequality also holds with $K^d_h(\cdot)$ in place of $\|\cdot\|_{\BMO^d}.$
\begin{lemma} 
\label{l01}
Suppose $h$ is a non-negative function on $[0,\infty)$ satisfying $h(0)=0$ and~\eqref{e1}.  Then\textup, for any cube $Q,$ function $\varphi\in L^1(Q),$ and integer $m\geqslant0,$
$$
K^d_{h,Q}(\varphi_m)\leqslant 2K^d_{h,Q}(\varphi).
$$
\end{lemma}
\begin{proof}
Take any $i\geqslant 0$ and $J\in D_i(Q).$ If $i\geqslant m,$ then $\varphi_m=\av{\varphi_m}J$ on $J$ and so 
$$
\av{h(|\varphi_m-\av{\varphi_m}J|)}J=\av{h(0)}J=0.
$$
If $i<m,$ then $\av{\varphi}J=\av{\varphi_m}J$ and, using~the triangle inequality~\eqref{e1}, we obtain
\begin{align*}
\av{h(|\varphi_m-\av{\varphi_m}J|)}J&\leqslant \av{h(|\varphi-\av{\varphi}J|)}J+\av{h(|\varphi-\varphi_m|)}J\\
&=\av{h(|\varphi-\av{\varphi}J|)}J+\sum_{L\in D_{m-i}(J)}\!\!\!2^{(i-m)n}\,\av{h(|\varphi-\varphi_m|)}L\\
&=\av{h(|\varphi-\av{\varphi}J|)}J+\sum_{L\in D_{m-i}(J)}\!\!\!2^{(i-m)n}\,\av{h(|\varphi-\av{\varphi}L|)}L\\
&\leqslant 2K^d_{h,Q}(\varphi).
\end{align*}
Taking the supremum over all $J$ and then over all $i$ yields the statement of the lemma.
\end{proof}
We will need this lemma to prove that the condition $K^d_{h,Q}(\varphi)<\infty$ implies that $\varphi\in\BMO^d(Q).$ However, if we already know that $\varphi\in\BMO^d,$ we can refine the estimate of Lemma~\ref{l01}, under some additional assumptions on $h.$
\begin{lemma}
\label{l01.5}
Let $h$ be an increasing\textup, concave function on $[0,\infty)$ such that $h(0)=0.$ Then\textup, for any cube $Q,$ function $\varphi\in\BMO^d(Q),$ and integer $m\geqslant 0,$
$$
K^d_{h,Q}(\varphi_m)\leqslant K^d_{h,Q}(\varphi) + h\big(\|\varphi-\varphi_m\|_{\BMO^d(Q)}\big).
$$
\end{lemma}
\begin{proof}
Since $\varphi_m$ is a bounded function, we have $\varphi-\varphi_m\in\BMO^d(Q).$ Now, for any $J\in D(Q),$
\begin{align*}
\av{h(|\varphi_m-\av{\varphi_m}J|)}J&\leqslant \av{h(|\varphi-\av{\varphi}J|)}J+\av{h(|(\varphi-\varphi_m)-\av{\varphi-\varphi_m}J|)}J\\
&\leqslant \av{h(|\varphi-\av{\varphi}J|)}J+h(\av{|(\varphi-\varphi_m)-\av{\varphi-\varphi_m}J|}J)\\
&\leqslant \av{h(|\varphi-\av{\varphi}J|)}J+h\big(\|\varphi-\varphi_m\|_{\BMO^d(Q)}\big).
\end{align*}
Here we have used, in sequence: \eqref{e1}, the concavity of $h,$ and the fact that $h$ is increasing. Now, the rightmost side of this inequality is bounded by $K^d_{h,Q}(\varphi)+h\big(\|\varphi-\varphi_m\|_{\BMO^d(Q)}\big)$ and, taking the supremum over all $J$ on the left, we obtain the statement of the lemma.
\end{proof}
\begin{remark}
The argument just given actually shows that 
$$
\big|K^d_{h,Q}(f)-K^d_{h,Q}(g)\big|\leqslant h\big(\|f-g\|_{\BMO^d(Q)}\big)
$$
for any two functions $f$ and $g$ for which the quantities involved are finite. The same inequality holds with $K_{h,Q}$ and $\BMO(Q)$ in place of $K^d_{h,Q}$ and $\BMO^d(Q),$ respectively. 
\end{remark}

We would now like to show that for $\varphi\in\BMO^d(Q)$ the functional $K^d_{h,Q}(\varphi)$ admits a non-trivial estimate from below in terms of  $\|\varphi\|_{\BMO^d(Q)}.$ To that end, take $t\geqslant0$ and let
\eq[dom]{
\Omega_t=\{x=(x_1,x_2)\in\mathbb{R}^2\colon x_1^2\leqslant x_2\leqslant x_1^2+t^2\}
}
and for each $x\in\Omega_t,$ 
\eq[bel0]{
E_{x,t,Q}=\{\varphi\colon~\,\,\,\varphi\hbox{~is dyadic-simple on }Q, \av\varphi Q =x_1, \av{\varphi^2}Q=x_2, \|\varphi\|_{\BMO^d(Q)} \leqslant t\}.
} 
We now define the following lower Bellman function:
\eq[bel]{
\bel{B}^d_t(x_1,x_2)=\inf\{\av{h(|\varphi|)}Q\!\colon\,\, \varphi\in E_{x,t,Q}\}. 
}
Note that this function does not depend on the choice of $Q$. It is easy to show that $E_{x,t,Q}$ is non-empty for any $t\geqslant0$ and any $x\in\Omega_t.$
Let $A(t)=\bel{B}^d_t(0,t^2)$. The following lemma is a direct consequence of the definition of $A.$
\begin{lemma}
\label{l1}
For any cube $Q$ and any dyadic-simple function $\varphi$ on $Q,$
\begin{align}\label{ineq2}
A(\| \varphi\|_{\BMO^d(Q)})\leqslant K^d_{h,Q}(\varphi).
\end{align}
\end{lemma}
\begin{proof}
Since $\varphi$ is dyadic-simple, it is in $\BMO^d(Q);$ let $t=\|\varphi\|_{\BMO^d(Q)}.$ Furthermore, there exists a dyadic subcube $J\subset Q$ such that $\av{(\varphi-\av{\varphi}J)^2}J=t^2$ and, therefore, $\|\varphi \|_{\BMO^d(J)}=t$. Let $\psi=\varphi|^{}_{J}-\av{\varphi}{J};$ then $\psi$ lies in the set $E_{(0,t^2),t,J}$~defined by~\eqref{bel0} (with $J$ in place of $Q$). Therefore,
 $$
 A(t)=\bel{B}^d_t(0,t^2)\leqslant \av{h(|\psi|)}J=\av{h(|\varphi-\av{\varphi}J|)}J\leqslant K^d_{h,Q}(\varphi).
 $$
\end{proof}
Our next result lies deeper; it is the key element in the proof of Theorem~\ref{t0}. Its proof is given in Section~\ref{Bellman_technique}. 
\begin{lemma}  
\label{bl}
If $h$ satisfies the conditions of Theorem~\ref{t0}\textup, then for any $t\ge0,$ 
\eq[main]{
A(t) \geqslant 2^{-(n+2)}h\big(2^{(n+2)/2} t\big).
}
\end{lemma}

The next lemma is the equivalent of Theorem~\ref{t0} for $\BMO^d$ and $K^d_h.$
 \begin{lemma} \label{t2}
If $h$ satisfies the conditions of Theorem~\ref{t0} and $\varphi$ is such that $K^d_{h,Q}(\varphi)<\infty,$ then $ \varphi \in \BMO^d(Q)$ and 
\begin{align} \label{ineq3}
 2^{-(n+2)}h\big(2^{(n+2)/2} \| \varphi\|_{\BMO^d(Q)}\big) \leqslant K^d_{h,Q}(\varphi).
\end{align}
 \end{lemma}
\begin{proof}
Setting $h(\infty)=\infty,$ and using, in order, the continuity of $h,$ Lemma~\ref{bl}, Lemma~\ref{l1}, and Lemma~\ref{l01}, we obtain
\begin{align*}
2^{-(n+2)}h(2^{(n+2)/2} \| \varphi\|_{\BMO^d(Q)})&=\lim\limits_{k \to \infty} 2^{-(n+2)}h(2^{(n+2)/2} \| \varphi_k\|_{\BMO^d(Q)}) \\
&\leqslant \limsup\limits_{k\to \infty} A(\|\varphi_k\|_{\BMO^d(Q)})\\
&\leqslant \limsup\limits_{k\to \infty}  K^d_{h,Q}(\varphi_k)\\
&\leqslant 2 K^d_{h,Q}(\varphi).  
\end{align*}
Therefore, $\varphi\in\BMO^d(Q),$ which immediately allows us to improve this estimate with the use of Lemma~\ref{l01.5}. We have
$$
A(\|\varphi_k\|_{\BMO^d(Q)})\le K^d_{h,Q}(\varphi_k)\le K^d_{h,Q}(\varphi)+h\big(\|\varphi-\varphi_k\|_{\BMO^d(Q)}\big).
$$
It is easy to show that $\varphi_k\to\varphi$ in the $\BMO^d(Q)$ norm, and so we can replace $2 K^d_{h,Q}(\varphi)$ above with $K^d_{h,Q}(\varphi).$
\end{proof}

We are now in the position to finish the proof of Theorem~\ref{t0}.
\begin{proof}[Proof of Theorem \ref{t0}]
 Take any subcube $J$ of $Q$. We have 
 $$
 K^d_{h,J}(\varphi)\leqslant K_{h,J}(\varphi)\leqslant K_{h,Q}(\varphi)<\infty,
 $$
 and so Lemma~\ref{t2} applies:
$$
2^{-(n+2)}h\Big(2^{(n+2)/2} \| \varphi\|_{\BMO^d(J)}\Big)\leqslant K^d_{h,J}(\varphi)\leqslant K_{h,Q}(\varphi).
$$
Taking supremum over all $J$ gives
$$
2^{-(n+2)}h\Big(2^{(n+2)/2} \| \varphi\|^{}_{\BMO(Q)}\Big)\leqslant K_{h,Q}(\varphi),
$$
which is equivalent to~\eqref{ms}.
\end{proof}
\section{Proof of Lemma~\ref{bl}}
\label{Bellman_technique}
To prove Lemma~\ref{bl}, we present a non-trivial {\it sub-solution} of the extremal problem~\eqref{bel}, i.e., a function $B$ on $\Omega_t$ such that 
\eq[u1]{
B(x)\leqslant\bel{B}^d_t(x), \quad\forall x\in\Omega_t,
}
and then show that
\eq[u2]{
B(0,t^2)= 2^{-(n+2)}h\big(2^{(n+2)/2} t\big).
}

Our sub-solution $B$ comes from the general Bellman function theory of integral estimates on BMO started in~\cite{svas} and developed further in~\cite{bel_ref_short,bel_ref_staraya}. To arrive at $B$, we first define a special family of functions that are locally convex on $\Omega_t$ (i.e., convex on every convex subset of $\Omega_t$) and then choose the largest element of that family for which we can establish~\eqref{u1} with minimal effort. The steps involved in constructing such locally convex or locally concave functions are beyond the scope of this paper; we refer the interested reader to~\cite{svas} and~\cite{bel_ref_staraya}. Here, we restrict ourselves to the simple verification of the fact that our chosen $B$ possesses properties~\eqref{u1} and~\eqref{u2}.  

We start by defining a function $\GG_t$ on the set  $\{x_1\geqslant0\}\cap \Omega_t$ as follows:
\begin{equation}\label{defG}
\GG_t(x_1,x_2)=
\begin{cases}
\frac{x_1^2}{x_2}h(\frac{x_2}{x_1}),& x_2< 2t x_1,\\
\\
\frac{x_2}{4t^2}h(2t),&  x_1\leqslant t, \, x_2 \geqslant 2tx_1,\\
\\
h(u)+(x_1-u)m(u),&  x_1 \geqslant t, \, x_2 \geqslant 2tx_1,
\end{cases}
\end{equation} 
where $u=u(x_1,x_2)=x_1+t-\sqrt{t^2-x_2+x_1^2}$, and the function
$m$ is the unique solution of the following Cauchy problem
$$
tm'(u)+m(u)=h'(u),  \qquad m(2t)=\frac{h(2t)}{2t}.
$$
We extend $\GG_t$ to all of $\Omega_t$ by symmetry: 
\eq[defG1]{
\GG_t(x_1,x_2)=\GG_t(-x_1,x_2)~\text{for}~x_1<0.
}

The reader can verify by direct calculation that $\GG_t$ defined by~\eqref{defG} and~\eqref{defG1} is locally convex on $\Omega_t$ and that $\GG_t(x_1,x_1^2)=h(|x_1|)$. This verification is somewhat harder in the part of $\Omega_t$ where $\GG_t$ is defined using the function $m.$ The reader is invited to consult Section~3.1 of~\cite{bel_ref_staraya} on how to deal with such difficulties. 

We omit the trivial proof of the following lemma.
\begin{lemma}
\label{l06}
Let $U=(u_1,u_2)$ and $V=(v_1,v_2)$ be two points in $\mathbb{R}^2$ such that $u_2\geqslant u_1^2$ and $v_2 \geqslant v_1^2$. Suppose that the midpoint $(U+V)/2$ lies in $\Omega_t$ for some $t>0$. Then the whole segment $[U,V]$ lies in $\Omega_{\sqrt{2}\,t}$. Thus\textup, if a function $G$ is locally convex on $\Omega_{\sqrt{2}\,t}$\textup, then 
$$
G\Big(\frac{U+V}2\Big) \leqslant \frac12G(U)+\frac12G(V).
$$
\end{lemma}
Our next result follows by a repeated application of Lemma~\ref{l06} to points formed by averages of a BMO function. In Bellman-function contexts this kind of argument is often referred to as ``Bellman Induction.''
\begin{lemma}
\label{l6}
Let $Q$ be a cube in $\mathbb{R}^n.$ Suppose $\varphi\in\BMO^d(Q)$ and $\|\varphi\|_{\BMO^d(Q)}\leqslant t.$ Let $G$ be a locally convex function on $\Omega_{2^{n/2} t}.$ Then\textup, for any $k\geqslant0,$
$$
G\big(\av{\varphi}Q,\av{\varphi^2}Q\big)\leqslant\sum_{J\in D_k(Q)} 2^{-nk}\,G\big(\av{\varphi}J,\av{\varphi^2}J\big).
$$ 
\end{lemma}
\begin{proof}
Let $J_1, J_2,\dots, J_{2^n}$ be an enumeration of the elements of $D_1(Q).$ Let 
$$
P_1=\bigcup_{j=1}^{2^{n-1}} J_j,\quad P_2=\bigcup_{j=2^{n-1}+1}^{2^n} J_j.
$$
Since $\|\varphi\|_{\BMO^d(Q)}\leqslant t,$ we have $(\av{\varphi}Q,\av{\varphi^2}Q)\in\Omega_t.$ Lemma~\ref{l06} now implies that both points $\big(\av{\varphi}{P^{}_1},\av{\varphi^2}{P^{}_1}\big)$ and $\big(\av{\varphi}{P^{}_2},\av{\varphi^2}{P^{}_2}\big)$ are in $\Omega_{\sqrt 2\,t}.$ Since $G$ is locally convex on $\Omega_{2^{n/2} t}$, it is also locally convex on $\Omega_{\sqrt2\,t}$ and thus
$$
G\big(\av{\varphi}Q,\av{\varphi^2}Q\big)\leqslant \frac12 G\big(\av{\varphi}{P^{}_1},\av{\varphi^2}{P^{}_1}\big)+ \frac12 G\big(\av{\varphi}{P^{}_2},\av{\varphi^2}{P^{}_2}\big).
$$
Now, write $P_1=R_1\cup R_2$ and $P_2=R_3\cup R_4,$ where each $R_i$ is a union of $2^{n-2}$ elements of $D_1(Q).$ By Lemma~\ref{l06}, we have $(\av{\varphi}{R_i},\av{\varphi^2}{R_i})\in\Omega_{2 t}$ and since $G$ is locally convex on $\Omega_{2t}$, we get
$$
G\big(\av{\varphi}Q,\av{\varphi^2}Q\big)\leqslant \frac14\sum_{i=1}^4 G\big(\av{\varphi}{R_i},\av{\varphi^2}{R_i}\big).
$$
Continuing in this fashion we conclude that 
\eq[z1]{
G\big(\av{\varphi}Q,\av{\varphi^2}Q\big)\leqslant \sum_{J\in D_1(Q)} \!\!\!2^{-n}\,G\big(\av{\varphi}J,\av{\varphi^2}J\big).
}
Since $\|\varphi\|_{\BMO^d(Q)}\leqslant t,$ each point $(\av{\varphi}J,\av{\varphi^2}J)$ is in $\Omega_t$, and so we can apply~\eqref{z1} again:
$$
G\big(\av{\varphi}Q,\av{\varphi^2}Q\big)\leqslant \sum_{J\in D_1(Q)} \!\!\!2^{-n}\sum_{L\in D_1(J)}\!\!\!2^{-n}\,G\big(\av{\varphi}L,\av{\varphi^2}L\big)=\sum_{J\in D_2(Q)}\!\!\!2^{-2n}\,G\big(\av{\varphi}J,\av{\varphi^2}J\big).
$$
Repeating this process $k-2$ more times yields the statement of the lemma.
\end{proof}

The final step of the proof now follows.
\begin{lemma}\label{L8}
For any $x \in \Omega_t$ the inequality
$
\bel{B}^d_t(x) \geqslant \GG_{2^{n/2} t}(x)
$
holds.
\end{lemma}
\begin{proof}
The function $\GG_{2^{n/2} t}$ is locally convex on $\Omega_{2^{n/2} t}$ by construction. Fix any $x\in \Omega_t$ and any dyadic-simple function $\phi$ on a cube $Q$ such that $\|\varphi\|_{\BMO^d(Q)}\leqslant t$ and $(\av{\phi}{Q},\av{\phi^2}{Q})=x$. For $k$ large enough $\varphi$ is constant on each cube $J\in D_k(Q)$, i.e., $\phi=\av{\varphi}J$ on $J$. Using Lemma~\ref{l6} and the boundary condition $\GG_{2^{n/2}t}(x_1,x_1^2)=h(|x_1|),$ we obtain
\begin{align*}
\GG_{2^{n/2}t}(x)&\leqslant \sum_{J\in D_k(Q)}\!\!\!2^{-nk}\,\GG_{2^{n/2}t}\big(\av{\phi}J,\av{\phi^2}J\big)\\
&=\sum_{J\in D_k(Q)}\!\!\!2^{-nk}\, \GG_{2^{n/2}t}\big(\av{\phi}J,\av{\phi}J^2\big)=
\sum_{J\in D_k(Q)}\!\!\!2^{-nk}\, h(|\av{\varphi}J|)=\av{h(|\varphi|)}Q.
\end{align*}
Taking the infimum over all such $\varphi,$ which is precisely the infimum in the definition of $\bel{B}_t^d,$ finishes the proof.
\end{proof}

Formula~\eqref{defG} gives $\GG_t(0,x_2)=\frac{x_2}{4t^2}\,h(2t)$ and so
$$
\GG_{2^{n/2} t}(0,t^2)=2^{-n-2}\,h(2^{(n+2)/2}\, t).
$$
Setting $B(x)=\GG_{2^{n/2} t}(x),~x\in\Omega_t,$ we see that both~\eqref{u1} and~\eqref{u2} are satisfied.
The proof of Lemma~\ref{bl} is now complete.\hfill $\qed$

\section{Proof of Theorem~\ref{t0.5}}
\label{pt0.5}
\begin{proof}
In dimension 1, we do not need to resort to dyadic splits when running the Bellman induction argument, as we did in Lemma~\ref{L8} above. Instead, Lemmas~6.2 and~6.4 of~\cite{svas} show how one can construct a quasi-dyadic system of subintervals of $Q$ and perform the induction without having to enlarge the domain $\Omega_t$ by a factor of $\sqrt2$ on every split.

Take any interval $Q$ and a function $\varphi$ on $Q$ such that $\kf<\infty.$ By Theorem~\ref{t0}, we have $\varphi\in\BMO(Q).$ Take $M>0$ and let $\varphi_M$ be the cut-off of $\varphi$ at height~$M:$ 
$
\varphi^{}_M=\varphi\,\chi^{}_{\{|\varphi|\le M\}}+M\,\chi^{}_{\{\varphi>M\}}-M\,\chi^{}_{\{\varphi<-M\}}.
$

Now recall the locally convex function $\GG_t$ defined on $\Omega_t$ by~\eqref{defG} and~\eqref{defG1}. Let $t=\|\varphi^{}_M\|^{}_{\BMO(Q)}.$ By Lemma~6.4 of~\cite{svas}, for any subinterval $J$ of $Q$ we have
$$
\GG_t\big(0,\av{\varphi^2_M}J-\av{\varphi^{}_M}J^2\big)\leqslant \av{h(|\varphi^{}_M-\av{\varphi^{}_M}J|)}J,
$$
or, using~\eqref{defG},
\eq[u3.1]{
\frac{\av{\varphi^2_M}J-\av{\varphi^{}_M}J^2}{4\|\varphi^{}_M\|^2_{\BMO(Q)}}\,h(2\|\varphi^{}_M\|^{}_{\BMO(Q)})\leqslant \av{h(|\varphi^{}_M-\av{\varphi^{}_M}J|)}J.
}
Arguing as in Lemma~\ref{l01.5}, we write
$$
\av{h(|\varphi^{}_M-\av{\varphi^{}_M}J|)}J \leqslant \av{h(|\varphi-\av{\varphi}J|)}J+h\big(\|\varphi-\varphi^{}_M\|^{}_{\BMO^d(Q)}\big).
$$

It is easy to show that $\av{\varphi^{}_M}J,$ $\av{\varphi^2_M}J,$ $\|\varphi^{}_M\|_{\BMO(Q)},$ and $\|\varphi-\varphi^{}_M\|_{\BMO(Q)}$ converge, respectively, to $\av{\varphi}J,$ $\av{\varphi^2}J,$ $\|\varphi\|_{\BMO(Q)},$ and $0,$ as $M\to\infty.$ Taking the limit, we obtain~\eqref{uu2}. To prove that this inequality is sharp, we fix an interval $Q$ and $t>0$ and present a function $\varphi\in\BMO(Q)$ with $\|\varphi\|_{\BMO(Q)}=t$ for which the inequality become equality. Without loss of generality, we can set $Q=[0,1].$ Consider the following function, which was constructed in~\cite{svas} to show sharpness in a similar situation:
$$
\varphi(s)=
\begin{cases}
-2t,&s\in\big[0,\frac18\big]\cup\big[\frac78,1\big],\\
0,&s\in\big(\frac18,\frac78\big),\\
2t,&s\in\big[\frac78,1\big].
\end{cases}
$$
We have $\|\varphi\|^{}_{\BMO(Q)}=t,$ $\av{\varphi}Q=0,$ $\av{\varphi^2}Q=t,$ and $\av{h(|\varphi-\av{\varphi}Q|)}Q=h(2t)/4,$ and thus both sides of~\eqref{uu2} are equal. 

To prove the right-hand inequality in~\eqref{uu3}, bound the right-hand side of~\eqref{uu2} by $\kf,$ take the supremum over all $J,$ and invert $h.$ 
\end{proof}

\section{Proof of Theorem~\ref{t1}}
\label{pt1}
\begin{proof}
The idea is very simple: we construct a function $\tilde{h}$ on $[0,\infty)$ such that $\tilde{h}$ satisfies the conditions of Theorem~\ref{t0} and, in addition, $\tilde{h}(x)\leqslant h(x)+C$ for all $x$ and some constant $C.$ Then
$$
K_{\tilde h,Q}(\varphi)<K_{h,Q}(\varphi)+C<\infty
$$ 
and so $\varphi\in\BMO(Q)$ by Theorem~\ref{t0}. The following lemma presents the construction.
\begin{lemma} \label{red1}
Let a function $f\colon [0,\infty)\to [3,\infty)$ be such that 
$$
f(t)\underset{t\to\infty}{\longrightarrow}\infty.
$$
Then there exists a smooth function
$\tilde f\colon[0,\infty)\to [0,\infty)$ such that
\eq[f1]{
\tilde{f} \leqslant f,\quad \tilde f(0)=0,\quad \tilde{f}(t)\underset{t\to\infty}{\longrightarrow}\infty,
}
and for all $t>0,$
\eq[f2]{
\tilde f'(t)>0,\quad
\tilde f''(t)<0,\quad
\tilde f'''(t)>0.
}
\end{lemma}
\begin{proof}
Define a sequence $\{t_m\}$ by  
$$
t_0=1 ,\quad t_m=\inf\{T\colon   T> 2t_{m-1} \text{~~and~~} f(t)\geqslant m,~\forall t>T \}.
$$
Now,
\eq[f3]{
\tilde{f}(t)= \sum_{m=3}^\infty (1- \exp(-t/t_m)),~t\geqslant0.
}
We have $t_m\geqslant 2^m,$ hence 
$$
1- \exp(-t/t_m)\leqslant \frac t{t_m}\leqslant\frac t{2^m}, 
$$
which means that the series in~\eqref{f3} converges uniformly in $t$ on any bounded subinterval of $[0,\infty).$ Likewise, the series of the derivatives of $1- \exp(-t/t_m)$ of any order converges uniformly. Differentiating term-wise, we readily obtain~\eqref{f2}.
The fact that $\tilde f(t)\underset{t\to\infty}{\longrightarrow}\infty$ follows since each summand in~\eqref{f3} is non-negative and increasing in $t$,  with limit $1$ as $t\to\infty.$

It remains to prove $\tilde f\leqslant f$. Suppose $t \in[t_{m}, t_{m+1})$ with $m \geqslant 3$, then $f(t)\geqslant m$. Note that 
 \begin{align*}
 \tilde f(t) &\leqslant \tilde f(t_{m+1})=\sum_{k=3}^\infty (1- \exp(-t_{m+1}/t_k) ) \\
& = \sum_{k=3}^m (1- \exp(-t_{m+1}/t_k) )+ \sum_{k=m+1}^\infty  (1- \exp(-t_{m+1}/t_k) )\\
&< m-2\,+\,\sum_{k=m+1}^\infty \frac{t_{m+1}}{t_k}\leqslant m-2\, +\,\sum_{k=m+1}^\infty 2^{m+1-k}=m\leqslant f(t).
 \end{align*}
 If $t \in [0, t_3)$, then 
 $$\tilde{f}(t) \leqslant \tilde{f}(t_3)\leqslant 3 \leqslant f(t). $$
\end{proof} 
To finish the proof of the theorem, we simply take $f=h+3$ in the lemma and let $\tilde h=\tilde f.$  
\end{proof}
\section{Is the condition $h(t)\mathop{\longrightarrow}\limits_{t\to \infty}\infty$ necessary?}
\label{nec}
We have shown that if
\eq[lim]{
h(t)\underset{t\to\infty}{\longrightarrow}\infty,
}
then $K_h$ implies $\BMO.$ By Lemma~\ref{t2}, the equivalent statement also holds in the dyadic case. A natural question arises: is~\eqref{lim} necessary for this implication to hold? If $h$ is additionally assumed to be increasing, the condition is trivially necessary, as otherwise $h$ is a bounded function. What can be said without that assumption? In this section, we give the answers for both $\BMO$ and $\BMO^d;$ the two cases turn out to be different. For the sake of simplicity we consider only $n=1$, but our constructions and proofs can be modified to fit any dimension.
\medskip

\subsection{The condition $\ds\lim_{t\to\infty}h(t)=\infty$ is not necessary for $K_h\Rightarrow \BMO$ }
~\medskip

We present a suitable function $h$ on $[0,\infty)$ for which~\eqref{lim} fails, but $K_h$ still implies $\BMO.$
Let $h$ be any continuous non-negative function such that
$$
h(t)=t^2 \quad\mbox{if}\quad t\in[0,1]\cup\Big(\bigcup_{k=1}^\infty \big[k+{\textstyle\frac14},k+{\textstyle\frac34}\big]\Big).
$$
At present we do not specify $h(t)$ for other $t;$ the following lemma works for all such $h.$
\begin{lemma}
If $Q$ is an interval and $\varphi\in L^1(Q)$ is such that $\kf<M$ for some $M>0,$ then $\varphi \in \BMO(Q)$ and $\|\varphi\|_{\BMO(Q)}\leqslant \sqrt{10M}.$ 
\end{lemma} 
\begin{proof}
Fix any subinterval $J \subset Q$. We have to show that 
$$\av{(\varphi-\av{\varphi}J)^2}J\leqslant10M.$$
Suppose not. Without loss of generality, we can assume $\av{\varphi}J=0$. Let 
\begin{align*}
A=\{x \in J\colon h(|\varphi(x)|)\ne \varphi^2(x)\},\quad A_+=\{x \in A\colon \varphi(x)>1\},\quad A_-=\{x \in A\colon \varphi(x)<-1\}.
\end{align*}
Thus, $A=A_+\cup A_-.$ We have
$$
10M|J|<\int_J\varphi^2=\int_{J\setminus A}h(|\varphi|)+\int_A \varphi^2.
$$
Since $\frac1{|J|}\int_{J\setminus A}h(|\varphi|)\leqslant \av{h(|\varphi-\av{\varphi}J|)}J\leqslant K_{h,Q}(\varphi)<M,$ we obtain
$$
\int_A \varphi^2>9M|J|,
$$
therefore, either 
\begin{equation}\label{contr}
\int_{A_+} \varphi^2>4M|J|\qquad\text{or}\qquad \int_{A_-} \varphi^2>4M|J|.
\end{equation}
It suffices to consider only the first case; the second one is completely symmetric. Since $\av{\varphi}J=0,$ by Klemes's version of the Riesz rising sun lemma (see~\cite{klemes}), there exists an at most countable set  $\{L_k\}$ of disjoint subintervals of $J$ such that $\av{\varphi}{L_k}=1/2$ and $\varphi\leqslant 1/2$ {\it a.e.} on $J\setminus \cup{L_k}.$ If $x \notin \cup L_k,$ then $\varphi(x) \leqslant 1/2,$ therefore $|A_+\setminus \cup L_k|=0.$
Now,
$$
|L_k|M\geqslant \int_{L_k}h\big(|\varphi-\av{\varphi}{L_k}|\big)=\int_{L_k}h\big(|\varphi-{\textstyle\frac12}|\big)\geqslant \int_{L_k\cap A_+}h\big(|\varphi-{\textstyle\frac12}|\big)=\int_{L_k\cap A_+}(\varphi-{\textstyle\frac12})^2. 
$$
The last equality uses the fact that if $x \in A_+$, then $|\varphi(x)-j|<1/4$ for some integer $j$ and so the fractional part $\{\varphi(x)-1/2\}\in(1/4,\,3/4).$ We now sum these inequalities over $k$ to get
$$
\int_{A_+}(\varphi-{\textstyle\frac12})^2\leqslant M\big(\sum |L_k|\big)\leqslant M|J|,
$$
which contradicts the first inequality in~\eqref{contr}, since on $A_+$ we have $\varphi-1/2>\varphi/2>0$ and thus 
$$
\int_{A_+}(\varphi-{\textstyle\frac12})^2>\frac14\int_{A_+}\varphi^2>M|J|.
$$
\vspace{-.6cm}

\end{proof}
Observe that we can easily ensure that $h(t)\underset{t\to\infty}{\centernot\longrightarrow}\infty$, for example by requiring that $h(k)=0$ for all integers $k>1.$
\medskip

\subsection{The condition $\ds\lim_{t\to\infty}h(t)=\infty$ is necessary for $K^d_h\Rightarrow\BMO^d$ }
~\medskip

The example of $h$ given in the previous section clearly does not work in the dyadic case. As shown below, no other example can work either, meaning that the limit condition on $h$ is necessary in this case.

\begin{lemma}
If $h$ is any non-negative function on $[0,\infty)$ for which~\eqref{lim} fails\textup, and $Q$ is any interval\textup, then there exists an integrable function $\varphi$ on $Q$ such that $K^d_{h,Q}(\varphi)<\infty$ while $\varphi\notin\BMO^d(Q).$ 
\end{lemma}
\begin{proof}
Without loss of generality, take $Q=(0,1).$ Since~\eqref{lim} fails, there exists a non-negative sequence $\{t_n\}$ and a number $M>0$ such that 
$$
t_n\underset{n\to\infty}{\longrightarrow}\infty\quad\text{~and~}\quad h(t_n) <M.
$$
By taking an appropriate subsequence, if needed, we may assume that each interval $[2^k,2^{k+1}]$ contains no more than one point $t_n.$ Now we can find a sequence of integers $\{n_j\}$ such that $2^{n_j}\leqslant t_j^2\leqslant 2^{n_j+1}.$ Define $\varphi$ on $Q$ by
$$
\varphi=\sum_{j=1}^\infty t_j h_{n_j},
$$
where $h_k$ is the $L^\infty$-normalized Haar function of the interval $J_k\df(2^{-k},2^{-k+1}),$ $h_k=\chi_{J_k^+}-\chi_{J_k^-}$ (here $J_k^-$ and $J_k^+$ are the left and right halves of $J_k,$ respectively). Observe that $\varphi\in L^1(Q):$
$$
\int_Q|\varphi|=\sum_{j=1}^\infty t_j\,2^{-n_j}\leqslant \sum_{j=1}^\infty 2^{(n_j+1)/2}\, 2^{-n_j}<\infty.
$$
Now, consider an interval $J\in D(Q).$ If $\varphi$ is not {\it a.e.} constant on $J,$ then $\av{\varphi}J=0$ and so $|\varphi-\av{\varphi}J|$ takes values in the set $\{t_1, t_2,...\}$ on $J.$ If $\varphi$ is constant on $J$, then $|\varphi-\av{\varphi}J|=0$ on $J.$ Therefore, $K^d_{h,Q}(\varphi)\leqslant\max\{h(0),M\}<\infty.$ However, $\varphi$ is not in $\BMO(Q)$ as it is not square-integrable on $Q:$
$$
\int_Q\varphi^2=\sum_{j=1}^\infty t^2_j\,2^{-n_j}=\infty.
$$
\end{proof}
\section*{Acknowledgment}
This research is supported in part by the Chebyshev Laboratory  (Department of Mathematics and Mechanics, St. Petersburg State University)  under the RF Government grant 11.G34.31.0026, and by JSC ``Gazprom Neft''. L.~Slavin is supported by the NSF (DMS-1041763). D.~M.~Stolyarov is supported by Rokhlin grant and the RFBR (grant 11-01-00526). V.~Vasyunin is supported by the RFBR (grant 11-01-00584-a).
P.~B.~Zatitskiy is supported by President of Russia grant for young researchers MK-6133.2013.1 and by the RFBR (grant 13-01-12422 ofi\_m2).
The authors are grateful to E.~Dubtsov for asking the main question considered in this work, and to the organizers of the summer school ``Bellman function method in harmonic analysis'' at Institut Mittag-Leffler in July 2013, where this project was conceived. Special thanks go to the referee for thoughtful comments and suggestions.

\end{document}